\documentclass[10pt]{article}
\usepackage{latexsym,color,amsmath,amsthm,amssymb,amscd,amsfonts}

\usepackage[style=numeric, giveninits,doi=false,isbn=false,url=false]{biblatex}
\newtheorem{lem}{Lemma}
\newtheorem{thm}[lem]{Theorem}
\newtheorem{prop}[lem]{Proposition}
\newtheorem{cor}[lem]{Corollary}

\newtheorem{rem}[lem]{Remark}
\newtheorem{conj}[lem]{Conjecture}

\def\RR{\mathbb{R}}
\def\vol{{\rm Vol}}
\def\conv{{\rm conv}}
\def\intr{{\rm int}}
\def\supp{{\rm supp}}

\begin{document}
\title {On unbalanced difference bodies and Godbersen's conjecture}
\author{S. Artstein-Avidan and E. Putterman}
\date{18 May 2025}
\maketitle

\begin{abstract}
    The longstanding Godbersen's conjecture states that for any convex body $K \subset \mathbb R^n$ of volume $1$ and any $j \in \{0, \ldots, n\}$, the mixed volume $V_j = V(K[j], -K[n - j])$ is bounded by $\binom
    {n}{j}$, with equality if and only if $K$ is a simplex. We demonstrate that several consequences of this conjecture are true: certain families of linear combinations of the $V_j$-s, arising from different geometric constructions, are bounded above by their values when one substitutes $\binom{n}{j}$ for $V_j$, with equality if and only if $K$ is a simplex. One of our results implies that for any $K$ of volume $1$ we have $\frac{1}{n + 1} \sum_{j = 0}^n \binom{n}{j}^{-1} V_j \le 1$, showing that Godbersen's conjecture holds ``on average'' for any convex body. Another result generalizes the well-known Rogers-Shephard inequality for the difference body.
\end{abstract}

\section{Introduction}
A convex body $K \subset \RR^n$ is a compact convex set with non-empty interior. 
For compact convex sets $K_1, \ldots,K_m \subset {\mathbb R}^n$, and non-negative real numbers $\lambda_1, \ldots,\lambda_m$, a classical result of Minkowski states that the volume of $\sum \lambda_i K_i$ is a homogeneous polynomial of degree $n$ in $\lambda_i$,
\begin{equation}\label{Eq_Vol-Def}
\vol \left(\sum_{i=1}^m \lambda_i K_i\right) = 
\sum_{i_1,\dots,i_n=1}^m \lambda_{i_1}\cdots\lambda_{i_n} V(K_{i_1},\dots,K_{i_n}). 
\end{equation}
The coefficient $V(K_{i_1},\dots,K_{i_n})$, which depends solely on $K_{i_1}, \ldots, K_{i_n}$, is called the mixed volume of $K_{i_1}, \ldots, K_{i_n}$.
The mixed volume is a non-negative, translation-invariant function, monotone with respect to set inclusion, invariant under permutations of its arguments, and positively homogeneous in each argument. For $K$ and $L$ compact and convex, we denote by $V(K[j], L[n-j])$ the mixed volume of $j$ copies of $K$ and $(n-j)$ copies of $L$. 
One has $V(K[n]) = \vol(K)$. By Alexandrov's inequality, $V(K[j],-K[n-j])\ge \vol(K)$, with equality if and only of $K= x_0-K$ for some $x_0 \in \mathbb R^n$, that is, if and only if $K$ is centrally symmetric. 
For further information on mixed volumes and their properties, see Section \textsection 5.1 of  \cite{schneider2014book}. 

A well-known inequality of Rogers and Shephard \cite{rogers1957difference,rogers1958convex} states that for any convex body $K \subset \mathbb R^n$, one has
$$\vol(K - K) \le \binom{2n}{n} \vol(K),$$
with equality if and only if $K$ is an $n$-simplex. Throughout the paper, we denote 
by $\Delta$ an $n$-simplex, which is arbitrary unless further specified. 

Expanding the left-hand side using \eqref{Eq_Vol-Def}, this can be rewritten as
\begin{equation}\label{eq:rs_mvol} \sum_{j=0}^n  {\binom{n}{j}} V(K[j], -K[n-j]) \le \binom{2n}{n} \vol(K).
\end{equation}
One can show \cite[\S 4]{rogers1957difference} that  $V(\Delta[j], -\Delta[n - j]) = \binom{n}{j} \vol(\Delta)$ for all $j \in \{0, \ldots, n\}$; substituting in \eqref{eq:rs_mvol} and noting that $\sum_{j = 0}^n \binom{n}{j}^2 = \binom{2n}{n}$, this shows that equality holds in \eqref{eq:rs_mvol} if $K = \Delta$.

In particular, one sees from the above line of argument that the Rogers-Shephard inequality would follow immediately if it were true that  $V(K[j], -K[n - j]) \le \binom{n}{j} \vol(K)$ for all $j$ and all $K$. This was conjectured in 1938 by Godbersen \cite{godbersen1938satz} (and independently by Hajnal and Makai Jr. \cite{Hajnal1974}):

\begin{conj}[Godbersen's conjecture]\label{conj:god}
For any convex body $K\subset \RR^n$ and any $1\le j\le n-1$,  
\begin{equation}\label{eq:Godbersen-conj} V(K[j], -K[n-j])\le \binom{n}{j} \vol(K),\end{equation}
with equality if and only if $K$ is an $n$-simplex.
\end{conj}

Godbersen \cite{godbersen1938satz} proved the conjecture for certain classes of convex bodies, in particular for those of constant width. Additionally, it is known that \eqref{eq:Godbersen-conj} holds for $j=1,n-1$ by the inclusion $K\subset n(-K)$ for bodies $K$ with center of mass at the origin, an inclusion which is tight for the simplex \cite{schneider2009stability}. The first author proved with Sadovsky and Sanyal that Godbersen's conjecture holds on the class of anti-blocking bodies \cite{ASS2023}; this result was later improved by Sadovsky \cite{sadovsky2024}, who showed that Godbersen's conjecture holds on the larger class of locally anti-blocking bodies (see the cited papers for the definitions of anti-blocking and locally anti-blocking bodies).

For $j$ other than $1, n - 1$, the best known bound on $V(K[j], -K[n -j])$ for general $K$ follows from a result of the first author with Einhorn, Florentin and Ostrover \cite{AEFO2015}, who showed that for any $\lambda \in [0,1]$ and for any convex body $K$ one has
\begin{equation}\label{eq:aaefo}
\lambda^j (1-\lambda)^{n-j} V(K[j], -K[n-j])\le {\vol(K)}.
\end{equation}
In particular, picking $\lambda = \frac{j}{n}$, we obtain
\[V(K[j], -K[n - j])\le \frac{n^n}{j^j (n-j)^{n-j}}\vol(K) \sim \binom{n}{j} \sqrt{2\pi \frac{j(n-j)}{n}} \vol(K). \] 

As Godbersen's conjecture seems out of reach in general for now, it seems reasonable to search for additional consequences of the conjecture which might be more tractable. We propose the following natural conjecture, which generalizes the Rogers-Shephard inequality to the \textit{unbalanced} difference body 
$$D_\lambda K = (1 - \lambda) K + \lambda (-K).$$ 

\begin{conj}[Unbalanced Rogers-Shephard] \label{conj:unbalancedRS}
	For any $\lambda\in (0,1)$ one has
\begin{equation}\label{eq:unb_rs0} \frac{\vol(D_\lambda K)}{\vol(K)} \le  \frac{\vol(D_\lambda \Delta)}{\vol(\Delta )}.
\end{equation}
	Moreover, equality holds if and only if $K$ is an $n$-simplex.
\end{conj}

If $\lambda = \frac{1}{2}$, then $D_\lambda K$ is just a scaling of the difference body $K - K$ and Conjecture \ref{conj:unbalancedRS} reduces to the Rogers-Shephard inequality. On the other hand, expanding both sides of \eqref{eq:unb_rs} in mixed volumes immediately yields that Conjecture \ref{conj:unbalancedRS} follows from Godbersen's conjecture. 

Explicitly, Conjecture \ref{conj:unbalancedRS} asks whether the following inequality holds for all $\lambda \in (0, 1)$:
\begin{equation}\label{eq:unb_rs}
\sum_{j=0}^n \binom{n}{j} \lambda^j (1-\lambda)^{n-j} V_j \le \sum_{j=0}^n \binom{n}{j}^2 \lambda^j (1-\lambda)^{n-j},  
\end{equation}
where we have denoted $V_j = V(K [j], -K[n-j])/\vol(K)$. 

Although we were unable to prove Conjecture \ref{conj:unbalancedRS} in general, we can show the following:

\begin{thm}\label{thm:unb_rs_45}
	For $n\le 5$, Conjecture \ref{conj:unbalancedRS} holds.
\end{thm}

As Godbersen's conjecture holds in dimension $n\le 3$, the novelty here is in dimensions $4$ and $5$.

Along with Theorem \ref{thm:unb_rs_45}, we prove two additional results, Theorems \ref{theorem-on-sum} and \ref{thm:gen_rs} below. Both results show that certain linear combinations of the normalized mixed volumes $V_j = V(K[j], -K[n - j])/\vol(K)$ of a convex body, similar to the combination in \eqref{eq:unb_rs}, are maximized precisely when $K$ is a simplex. Hence, both theorems would follow immediately from Godbersen's conjecture. What is less obvious is that Theorems \ref{theorem-on-sum} and \ref{thm:gen_rs} would also follow from the weaker Conjecture \ref{conj:unbalancedRS}. 

We now give the details. Our next result is an improvement of \eqref{eq:aaefo}:

\begin{thm}\label{theorem-on-sum}
For any convex body $K\subset \RR^n$ and any $\lambda \in [0,1]$ it holds that  
\begin{equation}\label{eq:thm_mvol_avg}
\sum_{j=0}^n \lambda^j (1-\lambda)^{n-j} V(K[j], -K[n-j])\le {\vol(K)},
\end{equation}
with equality if and only if $K$ is an $n$-simplex.
\end{thm}

The proof of Theorem \ref{theorem-on-sum} uses a geometric construction of a $(2n + 1)$-dimensional body whose projections and sections are related to Minkowski sums and to intersections, respectively, of certain families of dilates of $K$ and $-K$. This construction was introduced by Rogers and Shephard \cite{rogers1958convex}. 

As we stated above, \eqref{eq:thm_mvol_avg} would follow immediately from Godbersen's conjecture, by the binomial theorem. 

Theorem \ref{theorem-on-sum} has several interesting consequences. First, integration with respect to the parameter $\lambda$ yields:

\begin{cor}\label{cor-average-uniform} For any convex body $K\subset \RR^n$ it holds that 
\[ \frac{1}{n+1}
\sum_{j=0}^n \frac{V(K[j],-K[n-j])}{\binom{n}{j}}\le \vol(K),\] 
or equivalently
\[ \frac{1}{n-1}
\sum_{j=1}^{n-1} \frac{V(K[j],-K[n-j])}{\binom{n}{j}}\le \vol(K).\]
\end{cor}

So for any convex body $K$, Godbersen's conjecture for $K$ is true ``on average'' over all the mixed volumes. Of course, the fact that it holds true on average was known before, but with a different kind of average. Indeed, the Rogers-Shephard inequality \eqref{eq:rs_mvol} for the difference body can be rewritten as
\[ \frac{1}{\binom{2n}{n}} \sum_{j=0}^n  {\binom{n}{j}} V(K[j], -K[n-j]) \le \vol(K).\]

The new average in Corollary \ref{cor-average-uniform} is a uniform one, which gives us additional information. It implies, for instance, that the median of the sequence $( {\binom{n}{j}}^{-1}{V(K[j], -K[n-j])})_{j=1}^{n-1}$ is less than two, so that at least for one half of the indices $j=1,2,\ldots, n-1$, the mixed volumes satisfy Godbersen's conjecture up to a factor of $2$. More generally, apply Markov's inequality for the uniform measure on $\{1, \ldots, n-1\}$ to obtain:

\begin{cor}
	 For any convex body $K\subset \RR^n$  with $\vol(K)=1$, for at least $k$ of the indices $j=1,2,\ldots n-1$ it holds that \[  
	V(K[j],-K[n-j]) \le \frac{n-1}{n-k} \binom{n}{j}.\]
\end{cor}

Also, using the fact that $V(K[n]) = V(-K[n]) = \vol(K)$ and dividing through by $\lambda$, one can reformulate \eqref{eq:thm_mvol_avg} as
\[\sum_{j=1}^{n-1} \lambda^{j-1} (1-\lambda)^{n-j-1} \left[\frac{V(K[j], -K[n-j])}{\vol(K)} -\binom{n}{j}\right] \le 0,\]
so that by taking $\lambda=0, 1$ we see, once again, that 
 $V(K, -K[n-1]) = V(K[n-1],-K) \le n \vol(K)$; that is, our result includes the (known) case $j = 1,n-1$ of Godbersen's conjecture.

Finally, we state our third theorem, which is a generalization of the Rogers-Shephard inequality:

\begin{thm}\label{thm:gen_rs}  For any convex body $K\subset \RR^n$ and any $\lambda \in [0, 1]$  it holds that  
\begin{equation}\label{eq:gen_rs_eq}
\sum_{m = 0}^n \frac{(n!)^2}{(2n - m)! m!} (1 - \lambda)^m \lambda^{n - m} \sum_{j = 0}^{n - m} \binom{n - m}{j} V(K[n - j], -K[j]) \le \vol(K). 
\end{equation}
For any $\lambda$, equality is attained if and only if $K$ is an $n$-simplex.
\end{thm}

The Rogers-Shephard inequality \eqref{eq:rs_mvol} can be proven by analyzing the so-called covariogram function of $K$, defined via $g_K(x) = \vol(K \cap (x + K))$, which is supported on the difference body $K - K$. Similarly, Theorem \ref{thm:gen_rs} is proven by analyzing the function $g_{K, \lambda}(x) = \vol(K \cap (x + \lambda K))$, which is supported on the unbalanced difference body $K - \lambda K$. (In fact, it was in attempting to prove Conjecture \ref{conj:unbalancedRS} that we were led to Theorem \ref{thm:gen_rs}; though the unbalanced difference body does not appear in the statement of the result, it is crucial to the proof.) Obtaining sharp bounds on the integral of $g_{K, \lambda}$ requires introducing an additional ingredient, the local Steiner formula (equation \eqref{eq:loc_steiner} below), which to the best of our knowledge has not been previously used in proving Rogers-Shephard-type inequalities.

Note that writing $g(\lambda)$ for the LHS of \eqref{eq:gen_rs_eq}, we have $g(0) = \vol(K)$ and thus we must have $g'(0) \le 0$. Computing the derivative yields
$$-n\vol(K) + \frac{n}{n + 1} (\vol(K) + V(K[n - 1], -K[1])) \le 0,$$
i.e., $V(K[n - 1], -K[1]) \le n\vol(K)$, which again recovers the case $j = 1,n-1$ of Godbersen's conjecture.

As mentioned above, we shall show in section \ref{sec:unb_diff} that both our main results (Theorems \ref{theorem-on-sum} and \ref{thm:gen_rs}) would follow directly from Conjecture \ref{conj:unbalancedRS}, as the expressions on the right-hand sides of \eqref{eq:thm_mvol_avg} and \eqref{eq:gen_rs_eq} turn out to be integrals of $g(\lambda) = \vol(D_\lambda K)$ with respect to different measures on $[0, 1]$.


The structure of the paper is as follows. In section \ref{sec:unb_rs_45} we prove Theorem \ref{thm:unb_rs_45}. In section \ref{sec:rs_bodies} we prove Theorem \ref{theorem-on-sum}. In section \ref{sec:gen_rs} we prove Theorem \ref{thm:gen_rs}. The proof of the equality case of Theorem \ref{thm:gen_rs} uses a new characterization of the simplex, Proposition \ref{prop:hom_simp}, which may be of independent interest. In section \ref{sec:unb_diff} we show that Theorems \ref{theorem-on-sum} and \ref{thm:gen_rs} both follow from Conjecture \ref{conj:unbalancedRS}, and prove Theorem \ref{thm:unb_rs_45}. Finally, in section \ref{sec:extras} we give a simple geometric proof of an inequality which appeared earlier in \cite{AEFO2015}, using similar methods to those in the proof of Theorem \ref{theorem-on-sum}.

{\bf Acknowledgements:} We would like to thank Shay Sadovsky for helpful discussions, Maud Szusterman for providing the elegant proof of Proposition \ref{prop:hom_simp}, and the referee for carefully reading the paper and suggesting many improvements. Support for this work was provided by the ERC under the European Union’s Horizon 2020 research and innovation programme (grant agreement no. 770127), by ISF grant Number 784/20, and by the Binational Science Foundation (grant no. 2020329). In addition, part of this work was conducted while the second author was supported by a Chateaubriand Fellowship.

\section{Proof of Theorem \ref{thm:unb_rs_45}}\label{sec:unb_rs_45}

Fix a convex body $K \subset \RR^n$, and let $V_j = \frac{V(K[j], -K[n - j])}{\vol(K)}$. Recall that for any $j \in \{0, \ldots, n\}$, $V_{n - j} = V_j$. 

\begin{proof}[Proof of Theorem \ref{thm:unb_rs_45}]\mbox{}
We start with the case $n=4$. 
For  $K\subset \RR^4$, we have that $V_0 = V_4 = 1$ and $V_1 = V_3$. We wish to prove that for any $\lambda \in [0, 1]$ we have 
    $$\sum_{j = 0}^4  \binom{4}{j} \lambda^j (1 - \lambda)^{4 - j} V_j \le \sum_{j = 0}^4 \lambda^j (1 - \lambda)^{4 - j} \binom{4}{j}^2,$$
    which reduces to
	\[ 4(\lambda^3(1-\lambda)+\lambda(1-\lambda)^3)V_1 + 6\lambda^2 (1-\lambda)^2V_2 \le 
	16(\lambda^3(1-\lambda)+\lambda(1-\lambda)^3)  + 36\lambda^2 (1-\lambda)^2. \]
    As the case $\lambda \in \{0, 1\}$ is obvious, to prove the above inequality it is sufficient to divide through by $\lambda^2 (1 - \lambda)^2$ and prove the resulting inequality, 
    \[ 4\left(\frac{\lambda}{1 - \lambda}  + \frac{1 - \lambda}{\lambda}\right) V_1 + 6V_2 \le 
	16\left(\frac{\lambda}{1 - \lambda}  + \frac{1 - \lambda}{\lambda}\right)  + 36. \]
    Write $u = \frac{\lambda}{1 - \lambda} + \frac{1 - \lambda}{\lambda}$ and note that $u \ge 2$; we wish to prove that
    $$4u V_1 + 6 V_2 \le 16u + 36.$$
    The case $j = 1$ of Godbersen's conjecture implies $V_1 \le 4$, while the Rogers-Shephard inequality $\sum_{j = 0}^4 \binom{4}{j} V_j \le \binom{8}{4}$ reduces to $8V_1 + 6V_2 \le 68$. Multiplying the first inequality by $4(u - 2)$ and adding it to the second inequality we obtain
    $$4u V_1 + 6V_2 = 4(u - 2)V_1 + 8V_1 + 6V_2 \le 16(u - 2) + 68 = 16u + 36,$$
    which is precisely what we wished to prove. For the equality case, simply note that equality in both the Rogers-Shephard inequality and the inequality $V_1 \le 4$ is known to hold only if $K$ is a simplex.

\noindent   We turn to the case $n=5$, namely $K\subset \RR^5$. 
    The proof is exactly the same: we have $V_0 = V_5 = 1$, $V_1 = V_4$, $V_2 = V_3$, and the desired inequality reduces to 
    \begin{multline*} 5(\lambda^4(1-\lambda)+\lambda(1-\lambda)^4)V_1 + 10(\lambda^2 (1-\lambda)^3 + \lambda^3 (1 - \lambda)^2) V_2 \\ \le 
	25(\lambda^4(1-\lambda)+\lambda(1-\lambda)^4)  + 100 (\lambda^2 (1-\lambda)^3 + \lambda^3 (1 - \lambda)^2).
    \end{multline*}
    Write $\alpha = \lambda^4(1-\lambda)+\lambda(1-\lambda)^4$, $\beta = \lambda^2 (1-\lambda)^3 + \lambda^3 (1 - \lambda)^2$; we must show that
    \[5\alpha V_1 + 10 \beta V_2 \le 25\alpha + 100 \beta.\]
    One computes explicitly that $\alpha - \beta = \lambda (1 - \lambda) (1 - 2\lambda)^2 \ge 0$. Also, we know that $V_1 \le 5$, while the Rogers-Shephard inequality yields $5V_1 + 10 V_2 \le 125$. Adding $5(\alpha - \beta)$ times the first inequality to $\beta$ times the second, we obtain
    $$5\alpha V_1 + 10 \beta V_2 = 5(\alpha - \beta) V_1 + \beta (5 V_1 + 10 V_2) \le 25 (\alpha - \beta) + 125 \beta = 25\alpha + 100\beta,$$
    as desired. Again, the equality case follows immediately from the equality case in the Rogers-Shephard inequality.
\end{proof}

\section{Rogers-Shephard bodies and the proof of Theorem \ref{theorem-on-sum}}\label{sec:rs_bodies}

The proof of Theorem \ref{theorem-on-sum} proceeds via the consideration of two bodies, $C\subset \RR^{n+1}$ and  $T\subset \RR^{2n+1}$. Both were used in the paper of Rogers and Shephard \cite{rogers1958convex}. 

For $S \subset \mathbb R^n$, we use $\conv(S)$ to denote its convex hull. By imitating the methods of \cite{rogers1958convex}, we prove the following lemma:

\begin{lem}\label{lem:vol-of-C}
Given a convex body $K\subset \RR^n$ define $C\subset \RR\times \RR^n$ by 
\[ C = \conv (\{0\}\times(1-\lambda) K \cup \{1\}\times -\lambda K). \]
Then we have
\[ \vol(C)\le \frac{\vol(K)}{n+1},\]
and equality implies that $K$ is an $n$-simplex.
\end{lem}

With this lemma in hand, the proof of Theorem \ref{theorem-on-sum} is a simple computation:

\begin{proof}[Proof of Theorem \ref{theorem-on-sum}]
\begin{align*} \vol(C) =& \int_{0}^{1} \vol((1-\eta)(1-\lambda)K - \eta \lambda K) d\eta \\ =& 
 \sum_{j=0}^n \binom{n}{j}(1-\lambda)^{n-j}\lambda^j V(K[j],-K[n-j]) \int_{0}^{1}(1-\eta)^{n-j}\eta^j \, d\eta\\
 = & \frac{1}{n+1}\sum_{j=0}^n(1-\lambda)^{n-j}\lambda^j V(K[j],-K[n-j]).
\end{align*}
Thus, using Lemma \ref{lem:vol-of-C}, we have that 
$$\sum_{j=0}^n(1-\lambda)^{n-j}\lambda^j V(K[j],-K[n-j])\le \vol(K),$$
and equality implies that $K$ is an $n$-simplex. To show that if $K$ is a simplex then equality in fact holds, simply use the fact that $V(K[j],-K[n-j]) = \binom{n}{j}$ together with the binomial theorem.
\end{proof} 

%
%
%
%
%

A key ingredient in the proof of Lemma \ref{lem:vol-of-C} is the Rogers-Shephard inequality for sections and projections \cite{rogers1958convex} (for the full equality condition see e.g.~\cite{AEFO2015}):

\begin{lem}[Rogers and Shephard]\label{lem-RSsec-proj}
	Let $T\subset \RR^m$ be a convex body, let $E\subset \RR^m$ be a subspace of dimension $j$. Then
	\[ \vol(P_{E^{\perp}}T)\vol(T\cap E) \le \binom{m}{j}\vol(T), \]
	where $P_{E^{\perp}} $ denotes the projection operator onto $E^{\perp}$.
    Moreover, if equality holds then all the sections $\{(T \cap (E + x)^\perp): x \in P_{E^\perp}(T)\}$ are homothetic.
\end{lem}

We turn to the proof of Lemma \ref{lem:vol-of-C} regarding the volume of $C$.  

\begin{proof}[Proof of Lemma \ref{lem:vol-of-C}] 
We borrow directly the method of \cite{rogers1958convex}. Let $K_1,K_2\subset \RR^n$ be convex bodies. Consider the convex body $T\subset \RR^{2n+1}= \RR \times \RR^n \times \RR^n$ defined by  
\[ T = \conv (\{(0,0,y): y\in K_2\} \cup \{ (1,x,-x): x\in K_1\}). \]  
Written out in coordinates this is simply 
\begin{align*}  T &= \{ (\theta, \theta x , -\theta x + (1-\theta)y ): \theta \in [0, 1], x\in K_1, y\in K_2\}\\
& = \{ (\theta, w, z): \theta \in [0, 1], w\in \theta K_1, z+w\in (1-\theta)K_2\}. 
\end{align*}
The volume of $T$ is thus, by simple integration, equal to 
\[ \vol(T) = \vol(K_1) \vol(K_2) 
\int_0^1 \theta^n (1-\theta)^n\, d\theta = \frac{n!n!}{(2n+1)!}\vol(K_1) \vol(K_2). \] 
Fix $\theta_0 \in [0, 1]$. We consider the section of $T$ by the $n$-dimensional affine subspace 
\[ E = \{(\theta_0, x, 0): x\in \RR^n\}\] 
and project it onto the complement $E^\perp$. The section is given by
\[ T\cap E = \{ (\theta_0, x, 0): x\in \theta_0 K_1 \cap (1-\theta_0)K_2\}\] 
and so $\vol_n(T\cap E) = \vol(\theta_0 K_1 \cap (1-\theta_0) K_2 )$, while the projection is
\begin{align*} P_{E^{\perp}}T &= \{ (\theta, 0, y): \theta \in [0, 1], y \in \mathbb R^n\,|\, \exists x {\rm ~with~} (\theta, x, y)\in T \}\\
&=  \{ (\theta, 0, y): \theta \in [0, 1], \theta K_1 \cap ((1-\theta)K_2-y) \}\\
&=  \{ (\theta, 0, y):  \theta \in [0, 1], y\in  (1-\theta)K_2-\theta K_1 \}.
\end{align*}
Thus 
$$\vol_n(P_{E^\perp}T)= \vol(\{(\theta, y):  y\in  (1-\theta)K_2-\theta K_1\})$$ 
which is precisely a set of the type we considered before in $\RR^{n+1}$. In fact, taking $K_1 = \lambda K$, $K_2 = (1 - \lambda) K$ we get that $P_{E^\perp}T = C$.

Staying with our original $K_1$ and $K_2$, and using the Rogers-Shephard inequality for sections and projections (Lemma \ref{lem-RSsec-proj}), we see that 
\[ \vol(P_{E^{\perp}}T)\vol(T\cap E) \le \binom{2n+1}{n}\vol(T), \]
which translates to the following inequality   
\[ \vol (\conv (\{0\}\times K_2 \cup \{1\}\times (-K_1))) \le \frac{1}{n+1} \frac{\vol(K_1)\vol(K_2)}{\vol(\theta_0 K_1 \cap (1-\theta_0)K_2)}. \]

We mention that this exact same construction was performed and analyzed by Rogers and Shephard for the special choice $\theta_0 = 1/2$, which is optimal if $K_1 = K_2$. 

For our special choice of $K_2 = (1-\lambda)K$ and $K_1 = \lambda K$ we pick $\theta_0 = (1-\lambda)$ so that the intersection in question is simply $\lambda (1-\lambda) K$, which cancels out when we compute the volumes in the numerator. We end up with 
\[ \vol (\conv (\{0\}\times (1-\lambda)K \cup \{1\}\times (-\lambda K))) \le \frac{1}{n+1} \vol(K),  \] 
which was the statement of the lemma. 

It remains to consider the equality case. By the equality case of Lemma \ref{lem-RSsec-proj}, if equality holds then the sections \begin{align*}
    T_{\theta, y} &= T \cap \{(\theta, x, y): x \in \mathbb R^n\} \\
    &= \{(\theta, x, y): x \in \theta \cdot \lambda K, x + y \in (1 - \theta) \cdot (1 - \lambda) K \} \\
    &= (\theta, 0, y) + \{(0, x, 0): x \in \theta\lambda K \cap ((1 - \theta)(1 - \lambda) K)\}
\end{align*}
are all homothetic (when nonempty). In particular, taking $\theta = 1- \lambda$ as before we see that the bodies $\lambda(1 - \lambda) K \cap (\lambda(1 - \lambda) K + y)$ as $y$ varies are all homothetic, which implies by  \cite[Lemma 4]{rogers1957difference} that $\lambda (1 - \lambda) K$, and hence $K$, is a simplex.
\end{proof}

This concludes the proof of Lemma \ref{lem:vol-of-C}, and with it, the proof of Theorem \ref{theorem-on-sum}.

\section{The unbalanced covariogram and the proof of Theorem \ref{thm:gen_rs}}\label{sec:gen_rs}

Before proving Theorem \ref{thm:gen_rs}, we give two preliminary results.

\subsection{The local Steiner formula}
Let $K, B $ be convex bodies in $\RR^n$ such that  $0 \in \intr(B)$. For $x, y \in \mathbb R^n$, define
\begin{equation}\label{eq:gauge_dist}
d_B(x, K) = \min \{r \ge 0: x \in K + rB\}.
\end{equation}
It follows from \cite[(14.25), (14.27)]{SWstochastic} that if $K, B$ are in so-called general relative position (in particular, this holds if $K, B$ are $C^{2, +}$), one has the following local Steiner formula: for every nonnegative measurable function $g: [0, \infty) \to [0, \infty)$,
\begin{equation}\label{eq:loc_steiner}
\int_{\mathbb R^n} g(d_B(x, K))\,dx = \vol(K)g(0) + \sum_{m = 0}^{n - 1} V(K[m], B[n - m])  \cdot \binom{n}{m}\int_0^\infty (n - m) t^{n - 1 - m} g(t)\,dt.
\end{equation}
In fact, if $g$ is monotone, \eqref{eq:loc_steiner} holds for arbitrary convex bodies $K, B\subset \RR^n$ with $0 \in \intr(B)$. To see this, approximate $B$ from within and $K$ from without by sequences of smooth convex bodies $\{B_j\}_{j = 1}^\infty, \{K_j\}_{j = 1}^\infty$, and note that $d_{B_j}(x, K_j)$ is monotone increasing to $d_B(x, K)$, so $$\int_{\mathbb R^n} g(d_{B_j}(x, K_j))\,dx \to \int_{\mathbb R^n} g(d_B(x, K))\,dx$$
by monotone convergence. On the right-hand side, simply use that 
$$V(K_j[m], B_j[n - m]) \to V(K_j[m], B_j[n - m])$$ 
by the Hausdorff continuity of mixed volumes.

\subsection{A characterization of the simplex}
In order to prove the equality case in Theorem \ref{thm:gen_rs}, we need the following characterization of the simplex:

\begin{prop}\label{prop:hom_simp} Let $K \subset \mathbb R^n$ be a convex body, $\lambda \in (0, 1)$, and suppose that $(\lambda K + x) \cap K$ is homothetic to $K$ for all $x \in \intr(K - \lambda K)$. Then $K$ is an $n$-simplex.
\end{prop}

Many similar characterizations of the simplex are known; in particular, the case $\lambda = 1$, which we have already used above in the proof of Theorem \ref{theorem-on-sum}, yields the equality case in the Rogers-Shephard inequality \eqref{eq:rs_mvol} for $\vol(K - K)$. Not having found this particular characterization in the literature, we give a full proof.

As in \cite{rogers1957difference}, the proof of Proposition \ref{prop:hom_simp} consists of two steps: proving that $K$ must be a polytope, and showing that, given that $K$ is a polytope, it is a simplex. The first step in the proof is essentially the same as in \cite{rogers1957difference}. For the second step, surprisingly, there is a much simpler argument available in our setting than in the classical case.

\begin{proof}[Proof of Proposition \ref{prop:hom_simp}] First, we show that $K$ must be a polytope. We say that $p \in K$ is an exposed point of $K$ if and only if there exists $u \in S^{n - 1}$ such that 
$$\{p\} = K^u := \{x \in K: \langle x, u\rangle = \sup_{z \in K} \langle z, u\rangle\}.$$ 
It is known that $K$ is the closure of the convex hull of its exposed points \cite[Theorem 1.4.7]{schneider2014book}. Let $r$ be the inradius of $K$; we claim that any two exposed points of $K$ are at distance at least $2\lambda r$, which implies that there are only finitely many, and hence that $K$ is a polytope.

Indeed, suppose $x_1, x_2$ are exposed points of $K$ corresponding to vectors $u_1, u_2 \in S^{n - 1}$. If $|x_1 - x_2| < 2\lambda r$, one can translate $\lambda K$ by some vector $t \in K - \lambda K$ so that $x_1, x_2 \in \intr(\lambda K + t)$, and hence $K' = K \cap (\lambda K + t)$ is a convex body containing $x_1, x_2$; by assumption $K'$ must be of the form $\alpha K + z$ for some $\alpha < 1$ and $z \in \mathbb R$. For any $u \in S^{n - 1}$, $(K')^u = \alpha K^u + z$; in particular, $(K')^{u_i} = \{\alpha x_i + z\}$ for $i = 1, 2$ and so $d((K')^{u_1}, (K')^{u_2}) = \alpha |x_1 - x_2| < |x_1 - x_2|$. But since $x_1, x_2$ are exposed points of $K$, it is clear that they are also exposed points of $K' = K \cap (\lambda K + t)$, so $(K')^{u_1} = \{x_1\}$, $(K')^{u_2} = \{x_2\}$, contradiction.

It remains to show that a polytope $K$ satisfying the assumption must be a simplex. If $K$ is not a simplex, there exists a facet $F \subset K$ such that at least two vertices of $K$ lie outside $F$. (This is well-known, but we sketch a quick proof. Let $V$ denote the set of vertices of $K$, and suppose that for every facet $F \subset K$ we have $|V \backslash F| = 1$. Then $|V \backslash \bigcap_{i = 1}^k F_i| \le k$ for any facets $F_1, \ldots, F_k$ of $K$. As $K$ is an $n$-polytope, we may find $n$ facets $F_1, \ldots, F_n$ intersecting at a vertex, yielding $|V| = |V \backslash \bigcap_{i = 1}^n F_i| + 1 \le n + 1$, so $K$ is a simplex.)

So let $F = K^u$ be some facet of $K$ such that at least two vertices of $K$ do not belong to $F$, and let $v_1, v_2$ be any two such vertices. We fix the origin to be some point in the relative interior of $F$, so that $\lambda v_1, \lambda v_2 \in (\lambda K) \cap \intr(K)$. (It is crucial that $F$ is a facet; otherwise it would not follow that $\lambda v_i \in \intr(K)$, which we need.) Fix $\epsilon > 0$ sufficiently small so that $\lambda v_1 + \epsilon u, \lambda v_2 + \epsilon u \in K$. Then $K' = (\lambda K + \epsilon u) \cap K$ has strictly less volume than $\lambda K$, so there exists $\lambda' < \lambda$ such that $K' = \lambda' K + z$. This again yields a contradiction, as it implies that the vertices of $K'$ corresponding to $v_1, v_2$ under the homothety are at distance $\lambda' |v_1 - v_2|$, but on the other hand, these vertices must be $\lambda v_1 + \epsilon u, \lambda v_2 + \epsilon u \in K'$.
\end{proof}

\subsection{Proof of Theorem \ref{thm:gen_rs}}

\begin{proof}
Fix a convex body $K \subset \RR^n$, and define the function $f: \mathbb R^n \to [0, \infty)$ by
$f(x) = \frac{\vol( (\lambda K + x) \cap K)}{\vol (\lambda K)}$. The function $f$ is an ``unbalanced'' version of the covariogram function $g_K(x) = \vol (K \cap (K + x))$ which is often  used to prove the Rogers-Shephard inequality (see, e.g., \cite[\S 1.5.1]{AGA}). Our proof runs along the same lines, but is a bit more involved.

Let us list some useful properties of the function $f$. 
\renewcommand{\labelenumi}{(\roman{enumi})}
\begin{enumerate}
    \item By Fubini, one has 
    \begin{align*}
    \int_{\mathbb R^n} f &= \frac{1}{\vol (\lambda K)} \int_{\mathbb R^n} \vol ((\lambda K + x) \cap K)\,dx \\
    &= \frac{1}{\vol (\lambda K)} \int_{\mathbb R^n} \int_{\mathbb R^n} 1_K(y) 1_{\lambda K}(y - x) \,dy\,dx \\
    &= \frac{1}{\vol (\lambda K)} \int_{\mathbb R^n} \int_{\mathbb R^n} 1_K(y) 1_{\lambda K}(y - x) \,dx\,dy \\
    &=\frac{1}{\vol (\lambda K)}  \cdot \int_{\mathbb R^n} 1_K(y) \cdot \vol (\lambda K) \,dy = \vol( K).
    \end{align*}
    \item Clearly $f$ is supported on $\supp( f) = K - \lambda K$. 
    \item By a standard argument based on the Brunn-Minkowski inequality, $f$ is $\frac{1}{n}$-concave on its support, i.e., for all $y, z \in K - \lambda K$, $t \in [0, 1]$ one has 
    \begin{equation}\label{eq:cov_1/n_conc}
    f((1 - t)y + tz)^{\frac{1}{n}} \ge (1 - t) f(y)^{\frac{1}{n}} + t f(z)^{\frac{1}{n}}.
    \end{equation}
    To see this, use the fact that
    $$(\lambda K + ((1 - t)y + tz)) \cap K \supset (1 - t) \cdot ((\lambda K + y) \cap K) + t((\lambda K + z) \cap K),$$
    take the volume of both sides, and apply the Brunn-Minkowski inequality to the right-hand side. 
    
    Also note that, by the equality conditions of the Brunn-Minkowski inequality, equality in \eqref{eq:cov_1/n_conc} is possible only if $(\lambda K + y) \cap K$ and $(\lambda K + z) \cap K$ are homothetic.
    
    \item We have that $0\le f(x)\le 1$, that  $f(x) = 1$ if and only if $x \in (1 - \lambda) K$ and $f(x) = 0$ if and only if $x \not\in\intr ( K - \lambda K)$.  
\end{enumerate}

Note that any $x \in K - \lambda K$ can be written (generically in many ways) in the form $x = (1 - t) y + t z$ for $y \in (1 - \lambda) K$ and $z \in K - \lambda K$. Any such representation, by properties (iii) and (iv), yields a bound $f(x) \ge (1 - t)^n$. The best possible such bound is obtained by taking the minimal $t$ for which such a representation exists.

Let $K' = (1 - \lambda) K$, $B = \lambda (K - K)$, so that $K - \lambda K = K' + B$. One verifies that for any $x \in K' + B$, the minimal $t$ such that $x = (1 - t) y + t z$ for $y \in K'$ and $z \in K' + B$ is given by $d_B(x, K')$, as defined in \eqref{eq:gauge_dist}. In fact, all we need is the (a priori) weaker fact that if $d_B(x, K') \le r \le 1$ then there exist $y \in K'$, $w \in K' + B$ such that $x = (1 - r) y + r w$: but this is obvious, as if $x \in K' + rB$ then $x = y + rz$ for $y \in K'$, $z \in B$ and so $x = (1 - r) y + r(y + z)$, and $y + z \in K' + B$ by definition.

The upshot of all the above is that for any $x \in \mathbb R^n$, $f(x) \ge (1 - d_B(x, K'))^n 1_{K' + B}(x)$, and so 
\begin{equation}\label{eq:unb_cov_int}
    \vol(K)= \int_{\mathbb R^n} f \ge \int_{K' + B} (1 - d_B(x, K'))^n \,dx.
\end{equation}
To compute the right-hand side of \eqref{eq:unb_cov_int}, we apply \eqref{eq:loc_steiner} with $g(t) = (1 - t)^n 1_{t \le 1} $:
\begin{align*}\int_{K' + B} (1 - d_B(x, K'))^n \,dx &= |K'| + \sum_{m = 0}^{n - 1} V(K'[m], B[n - m])  \cdot (n - m)\binom{n}{m} \int_0^1 t^{n - 1 - m}  (1 - t)^n\,dt \\
&= |K'| + \sum_{m = 0}^{n - 1} V(K'[m], B[n - m]) \cdot (n - m)\binom{n}{m} \cdot \frac{\Gamma(n - m) \Gamma(n + 1)}{\Gamma(2n - m + 1)} \\
&=\sum_{m = 0}^n V(K'[m], B[n - m]) \cdot \frac{n! n!}{(2n - m)! m!}.
\end{align*}
Since $K' = (1 - \lambda) K$, $B = \lambda (K - K)$, we can expand the terms $V(K'[m], B[n - m])$ in mixed volumes:
\begin{align*}V(K'[m], B[n - m]) &= (1 - \lambda)^m \lambda^{n - m} V(K[m], (K - K)[n - m]) \\
&= (1 - \lambda)^m \lambda^{n - m} \sum_{j = 0}^{n - m} \binom{n - m}{j} V(K[n - j], -K[j])
\end{align*}
Combining the last three equations yields
$$\vol(K) \ge \sum_{m = 0}^n \frac{n! n!}{(2n - m)! m!} (1 - \lambda)^m \lambda^{n - m} \sum_{j = 0}^{n - m} \binom{n - m}{j} V(K[n - j], -K[j]),$$
which is \eqref{eq:gen_rs_eq}.

To show that equality holds for an $n$-simplex $\Delta$, recall that $V(\Delta[n - j], -\Delta[r]) = \binom{n}{n - j} \vol(\Delta)$ for all $j \in \{0, \ldots, n\}$. Thus, the interior sum on the right-hand side in \eqref{eq:gen_rs_eq} becomes
$$\vol(\Delta) \cdot \sum_{j = 0}^{n - m} \binom{n - m}{j} \binom{n}{n - j},$$
which by the Chu-Vandermonde identity simplifies to $\binom{2n - m}{n} \vol(\Delta)$, so the right-hand side is just
$$\vol(\Delta) \cdot \sum_{m = 0}^n \frac{n! n!}{(2n - m)! m!} (1 - \lambda)^m \lambda^{n - m} \binom{2n - m}{n} = \vol(\Delta)\sum_{m = 0}^n \binom{n}{m} (1 - \lambda)^m \lambda^{n - m} = \vol(\Delta),$$
as desired.

Finally, we show the converse, namely that equality in \eqref{eq:gen_rs_eq} holds \textit{only} for a simplex. First of all, equality in \eqref{eq:gen_rs_eq} implies equality in \eqref{eq:unb_cov_int}, so for almost every $x \in K' + B$ we have
\begin{equation}\label{eq:1/n_affine}
    f(x) = (1 - d_B(x, K'))^n,
\end{equation}
and hence this holds for every such $x$ because $f(x) =\vol( (\lambda K + x) \cap K)$ is continuous.

Let $x \in \intr(K' + B)$ and write $x = u + t v$ where $u \in \partial K'$, $v \in \partial B$, and $t = d_B(x, K') < 1$. (This decomposition need not be unique.) It is easy to see that $d_B(u + sv, K') = s$ for any $s \in [0, t]$, so \eqref{eq:1/n_affine} shows that along the segment $[u, x]$, the function $y \mapsto |(\lambda K + y) \cap K|^{1/n}$ is affine, which means that we must have equality in \eqref{eq:cov_1/n_conc}. By the remark below \eqref{eq:cov_1/n_conc}, this implies that for any $y \in [u, x]$ the sets $(\lambda K + x) \cap K$ and $(\lambda K + y) \cap K$ are homothetic. In particular, letting $y = u$ and noting that $u \in (1 - \lambda) K$, $(\lambda K + u) \cap K = \lambda K + u$, we see that $(\lambda K + x) \cap K$ is homothetic to $K$. 

In sum, we have obtained that equality in \eqref{eq:gen_rs_eq} implies that $(\lambda K + x) \cap K$ is homothetic to $K$ for all $x \in \intr(K - \lambda K)$. By Proposition \ref{prop:hom_simp}, this implies that $K$ is a simplex, concluding the proof of Theorem \ref{thm:gen_rs}.
\end{proof}

\section{Proofs of Theorems \ref{theorem-on-sum} and \ref{thm:gen_rs} assuming Conjecture \ref{conj:unbalancedRS}}\label{sec:unb_diff}
Throughout the section, we fix $K$ and write $\Delta$ for the $n$-simplex of volume $\vol(K)$. Recall that we use the notation $D_\lambda K = (1 - \lambda) K - \lambda K$ for the unbalanced difference body.

The goal in this brief section is to show that Theorems \ref{theorem-on-sum} and \ref{thm:gen_rs} would follow from the conjectured unbalanced Rogers-Shephard inequality (Conjecture \ref{conj:unbalancedRS}); that is, if we knew that $\frac{|D_\lambda K|}{|K|} \le \frac{|D_\lambda \Delta|}{|\Delta|}$ for all $\lambda \in [0, 1]$, it would directly imply Theorems \ref{theorem-on-sum} and \ref{thm:gen_rs}.

For Theorem \ref{theorem-on-sum}, recall that the main lemma in the proof (Lemma \ref{lem:vol-of-C}) is that $\vol(C) \le \frac{\vol(K)}{n + 1}$, where 
\[ C = \conv (\{0\}\times(1-\lambda) K \cup \{1\}\times -\lambda K). \]
In the proof of Theorem \ref{theorem-on-sum} using Lemma \ref{lem:vol-of-C}, we wrote 
$$\vol(C) = \int_{0}^{1} \vol((1-\eta)(1-\lambda)K - \eta \lambda K)\, d\eta$$ 
and then expanded the integrand in mixed volumes of $K$ and $-K$; but we could also have written $(1-\eta)(1-\lambda)K - \eta \lambda K = \alpha \cdot D_\mu K$ where $\alpha = (1 - \eta)(1 - \lambda) + \eta \lambda$ and $\mu = \frac{\eta \lambda}{\alpha}$, so that 
$$\vol((1-\eta)(1-\lambda)K - \eta \lambda K) = \alpha^n \vol(D_\mu K),$$ obtaining $\vol(C)$ as an integral over $f(\nu) = \vol(D_\nu K)$ with respect to some measure on $[0, 1]$ which is independent of $K$. Thus, if Conjecture \ref{conj:unbalancedRS} were true we would have
$\vol(C) \le \vol(C_\Delta)$ where $\Delta$ is a simplex with $\vol(\Delta)= \vol(K)$ and $$C_\Delta = \conv (\{0\}\times(1-\lambda) \Delta \cup \{1\}\times -\lambda \Delta).$$ 
By direct computation, or by using the fact that Theorem \ref{theorem-on-sum} is sharp for the simplex, we have $\vol(C_\Delta) = \frac{\vol(\Delta)}{n + 1}$, yielding Theorem \ref{theorem-on-sum} (without relying on Lemma \ref{lem:vol-of-C}). 

As for Theorem \ref{thm:gen_rs}, recall that we obtained the inequality 
\begin{equation}\label{eq:unb_covar}
    \int_{K + B} (1 - d_B(x, K'))^n\,dx \le \vol(K)
\end{equation}
where $K' = (1 - \lambda) K$, $B = \lambda (K - K)$. We rewrote the integral in terms of mixed volumes, but one can also use the identity $\int g = \int_0^\infty \vol (\{g \ge t\}) \,dt$ with $g = (1 - d_B(x, K'))^n$, 
and note that 
$$\{g \ge t\} = \{x: (1 - d_B(x, K'))^n \le t\} = K' + (1 - t^{1/n}) B,$$
which is again of the form $\alpha \cdot D_\lambda K$ for some $\alpha, \lambda$ depending on $t$. Thus, Theorem \ref{thm:gen_rs} also reduces to proving an inequality of the form $\int_0^1 w(\lambda) D_\lambda(K)\,d\lambda \le \vol(K)$
for some weight function $w: [0, 1] \to \mathbb R^+$ independent of $K$. Again, an explicit computation, or the fact that Theorem \ref{thm:gen_rs}
is sharp for the simplex, shows that 
$\int_0^1 w(\lambda) D_\lambda(\Delta)\,d\lambda =  \vol(\Delta)$,
so that if Conjecture \ref{conj:unbalancedRS} were true it would imply immediately that 
$\int_0^1 w(\lambda) D_\lambda(K)\,d\lambda \le \vol(K)$, yielding Theorem \ref{thm:gen_rs} while bypassing the argument based on the Brunn-Minkowski inequality that we used to derive \eqref{eq:unb_covar}.

\section{A new proof of an inequality from \cite{AEFO2015}}\label{sec:extras}

	We end this note with a simple geometric proof of the following inequality from \cite{AEFO2015} (which reappeared independently in \cite{agjv2016}):
 
	\begin{thm} \label{thm:strange} Let  $K,L \subset {\mathbb R}^n$ be convex bodies with $0 \in \intr(K \cap L)$. Then 
		\begin{equation}\label{eq:polar_sum_ineq}
		\vol (\conv(K \cup - L) ) \,
		\vol ((K^\circ +  L^\circ)^\circ ) \le
		\vol(K) \, \vol(L).
        \end{equation}
	\end{thm}
	We remark that this inequality can be thought of as a dual to the Milman-Pajor inequality \cite{milman2000} stating that when $K$ and $L$ have center of mass at the origin one has 
	\[ \vol (\conv(K \cap - L) ) \,
			\vol (K +  L) \ge
			\vol(K) \, \vol(L).\]

	\begin{proof}[Simple geometric proof of Theorem \ref{thm:strange}]
	Given $K, L \subset \mathbb R^n$, construct the body
	\[ C = \conv (K\times \{0\}\cup \{0\}\times L) \subset \mathbb R^{2n}.\]
	The volume of $C$ is  simply 
	\begin{equation}\label{eq:orth_conv_vol}
    \vol(C) = \vol(K) \vol(L) \frac{1}{\binom{2n}{n}}.
    \end{equation}
	Consider the two orthogonal subspaces of $\RR^{2n}$ of dimension $n$ given by $E = \{ (x,x): x\in \RR^n\}$ and $E^{\perp} = \{ (y,-y): y \in \RR^{n}\}$. First we compute $C\cap E$: 
	
	\[ C\cap E = \{ (x,x): x = \lambda y, x=(1-\lambda)z,\lambda \in [0,1], y\in K, z\in L\}. \] In other words, 
	\[ C\cap E = \{ (x,x): x\in \cup_{\lambda\in [0,1]}( \lambda K \cap (1-\lambda)L)\}= 
	 \{ (x,x): x\in (K^{\circ}+L^{\circ})^{\circ}\}.
	\] 
	Next let us calculate the projection of $C$ onto $E^{\perp}$: 
	Since $C$ is a convex hull, we may project $K\times \{0\}$ and $\{0\}\times L$ onto $E^{\perp}$ and then take a convex hull. In other words we are searching for all $(x,-x)$ such that there exists $(y,y)$ with $(x+y, -x+y)$ in  $K\times \{0\}$ or 
	$\{0\}\times L$. Clearly this means that $y$ is either $x$, in the first case, or $-x$, in the second, which yields 
	\[P_{E^{\perp}}C = \conv \{ (x,-x): 2x\in K {\rm ~or~}
	-2x \in L\} = \{ (x,-x): x\in \conv (K/2 \cup -L/2) \}.\]

	In terms of volume we get that 
	\[ \vol_n(C\cap E) = \sqrt{2}^n \vol_n((K^{\circ}+L^{\circ})^{\circ}) \] and 
	\[ \vol_n (P_{E^{\perp}}(C)) = \sqrt{2}^{-n} \vol_n (\conv (K \cup -L))\]
	and so their product is precisely the quantity on the right hand side of \eqref{eq:polar_sum_ineq}. On the other hand, by the Rogers Shephard inequality for sections and projections (Lemma \ref{lem-RSsec-proj}), we know that
	\[ \vol_n(C\cap E)\vol_n (P_{E^{\perp}}(C))
	\le \vol_{2n}(C)\binom{2n}{n}. \]
	Substituting $\vol(C)$ on the right-hand side using \eqref{eq:orth_conv_vol} we obtain the desired inequality \eqref{eq:polar_sum_ineq}.
    \end{proof}
	
	\begin{rem}
    The same method also yields another proof of \cite[Theorem 1.5]{AEFO2015}, which states that for any $\lambda \in [0, 1]$ and any $K$ containing the origin one has
    $$\conv((1-\lambda)K \cup -\lambda K) \le \vol(K).$$
	Indeed, taking $K = L$ in the construction of $C$, and choosing the subspace $E_\lambda = \{(\lambda x, (1-\lambda)x): x\in \RR^n\}$ in place of $E$, we obtain 
	\[ C\cap E_\lambda = \{ (\lambda x, (1-\lambda)x): x\in K\}\] and 
	\[ P_{E_\lambda^\perp} C = \left\{ ((1-\lambda)x, -\lambda x): x \in \frac{1}{\lambda^2 + (1-\lambda)^2} \cdot \conv((1-\lambda)K \cup -\lambda K)
	\right\}. \]
	As above, the product of the volumes of these sets, which is simply 
	\[ \vol( \conv((1-\lambda)K \cup -\lambda K))\vol(K),\]
	is bounded above by $\binom{2n}{n}\vol(C) = \vol(K)$. Note that in this proof, all the sets of the form ${(1-\lambda)K \cup -\lambda K}$ are realized as projections of a single fixed body, namely $C$. 
	\end{rem}

\printbibliography

{\small
\noindent S. Artstein-Avidan and E. Putterman,
\vskip 2pt
\noindent School of Mathematical Sciences, Tel Aviv University, Ramat
Aviv, Tel Aviv, 69978, Israel.\vskip 2pt
\noindent Email address: shiri@tauex.tau.ac.il, putterman@mail.tau.ac.il}

\end{document}